\documentclass[UTF8,11pt]{ctexart}
\usepackage[a4paper,margin=2.8cm]{geometry}
\usepackage{amsmath,amssymb,amsthm,mathtools,mathrsfs,bm}
\usepackage{iftex}
\ifPDFTeX
  \usepackage[T1]{fontenc}
\fi
\IfFileExists{newtxtext.sty}{\usepackage{newtxtext}}{}
\IfFileExists{newtxmath.sty}{\usepackage{newtxmath}}{}
\usepackage{microtype}
\usepackage{enumitem}
\usepackage{setspace}
\usepackage{hyperref}
\usepackage[nameinlink,noabbrev]{cleveref}
\usepackage{aliascnt}
\usepackage{fancyhdr}
\usepackage{lastpage}
\hypersetup{colorlinks=true,linkcolor=blue,citecolor=blue,urlcolor=blue}

\numberwithin{equation}{section}
\allowdisplaybreaks

\theoremstyle{plain}
\newtheorem{theorem}{Theorem}[section]

\newaliascnt{lemma}{theorem}
\newtheorem{lemma}[lemma]{Lemma}
\aliascntresetthe{lemma}

\newaliascnt{proposition}{theorem}
\newtheorem{proposition}[proposition]{Proposition}
\aliascntresetthe{proposition}

\newaliascnt{corollary}{theorem}
\newtheorem{corollary}[corollary]{Corollary}
\aliascntresetthe{corollary}

\theoremstyle{definition}
\newaliascnt{definition}{theorem}
\newtheorem{definition}[definition]{Definition}
\aliascntresetthe{definition}

\newaliascnt{example}{theorem}
\newtheorem{example}[example]{Example}
\aliascntresetthe{example}

\theoremstyle{remark}
\newaliascnt{remark}{theorem}
\newtheorem{remark}[remark]{Remark}
\aliascntresetthe{remark}

\theoremstyle{definition}
\newaliascnt{question}{theorem}
\newtheorem{question}[question]{Question}
\aliascntresetthe{question}

\Crefname{question}{question}{questions}
\Crefname{question}{Question}{Questions}

\crefname{theorem}{theorem}{theorems}
\Crefname{theorem}{Theorem}{Theorems}
\crefname{lemma}{lemma}{lemmas}
\Crefname{lemma}{Lemma}{Lemmas}
\crefname{proposition}{proposition}{propositions}
\Crefname{proposition}{Proposition}{Propositions}
\crefname{corollary}{corollary}{corollaries}
\Crefname{corollary}{Corollary}{Corollaries}
\crefname{definition}{definition}{definitions}
\Crefname{definition}{Definition}{Definitions}
\crefname{example}{example}{examples}
\Crefname{example}{Example}{Examples}
\crefname{remark}{remark}{remarks}
\Crefname{remark}{Remark}{Remarks}

\newcommand{\R}{\mathbb{R}}
\newcommand{\Z}{\mathbb{Z}}

\newcommand{\supp}{\operatorname{supp}}

\newcommand{\papertitle}{Common Spectral Eigenvalue Spectra for Random Convolutions \\Generated by Hadamard Triples}
\title{\papertitle}
\date{}

\newcommand{\paperauthors}{Wei-Jie Wang$^{1}$\quad Xin Yang$^{2,*}$}

\newcommand{\paperkeywords}{Spectral measure; Hadamard triples; Spectral eigenvalue; Beurling dimension.}

\newcommand{\paperaddressone}{Hubei Key Laboratory of Mathematical Sciences, College of Mathematics and Statistics, Central China Normal University, Wuhan, Hubei 430079, China}
\newcommand{\paperemailone}{wwjmath@163.com}
\newcommand{\paperaddresstwo}{School of Mathematics and Statistics, Huazhong University of Science and Technology, Wuhan 430074, China}
\newcommand{\paperemailtwo}{xyang@hust.edu.cn}

\begin{document}

\begin{center}
{\Large\bfseries \begin{tabular}{@{}c@{}}\papertitle\end{tabular}\par}
\vspace{0.9em}
{\large \paperauthors\par}
\end{center}
\vspace{0.7em}

\begin{abstract}
In a previous work by Lu \cite{Luzheng01}, it was proved that the set 
$$\mathcal{T}:=\{t\in\mathbb{Z}\setminus\{0\}:(q,\mathcal{D},tL)\text{ forms a Hadamard triple}\}$$
constitutes a spectral eigenvalue set for $\mu_{q,D}$, where $(q,\mathcal{D},L)$ is a Hadamard triple. And they prove for $s \in [0,\frac{\log \#\mathcal{D}}{\log q}]$, the corresponding spectra form a family of cardinality continuum. In this paper, we study Moran measures formed by random convolutions of finite Hadamard triples. Let 
$$\mu=\delta_{M_1^{-1}D_1}*\delta_{M_2^{-1}D_2}*\cdots,
 \qquad M_k=q_1q_2\cdots q_k,$$
where the factors are produced from finitely many triples
$\{(N_j,B_j,L_j):1\le j\le m\}$,
$(\omega_k)_{k=1}^{\infty}\in\{1,2,\ldots,m\}^{\mathbb N}$,
and $n_k\in\mathbb N^+$, by setting
$$
q_k=N_{\omega_k}^{n_k},\qquad D_k=B_{\omega_k},\qquad
E_k=N_{\omega_k}^{n_k-1}L_{\omega_k}.
$$
Assume a non-full-digit gap
$$\rho:=\min_{1\le j\le m}\frac{N_j}{\#B_j}>1.$$
For the common Hadamard triple multiplier set
$$\mathcal{T}_*:=\bigcap_{j=1}^m\{t\in\mathbb{Z}\setminus\{0\}:(N_j,B_j,tL_j)\text{ is a Hadamard triple}\},$$
Our main result is that, for every
$$0\le s\le \kappa_\omega:=\limsup_{R\to\infty}\frac{\sum_{r=1}^R\log \#D_r}{\sum_{r=1}^R\log q_r},$$
there exist continuum many countable sets $\Lambda\subset\mathbb{Z}$ such that $t\Lambda$ is a spectrum of $\mu$ for every $t\in\mathcal{T}_*$ and $\dim_{Be}\Lambda=s$.

\end{abstract}

\begingroup
\renewcommand{\thefootnote}{}
\footnotetext{%
\par
Corresponding author $*$.}
\endgroup

\section{Introduction}
A Borel probability measure $\mu$ on $\mathbb{R}^d$ is called a spectral measure if there exists a countable set $\Lambda \subset \mathbb{R}^d$, called a spectrum for $\mu$, such that the family of exponential functions $E(\Lambda):=\{e^{2\pi i \lambda x}:\lambda \in \Lambda\}$ forms an orthonormal basis for $L^2(\mu)$. If a spectral measure $\mu$ is the Lebesgue measure on a measurable set $\Omega$, then we say that $\Omega$ is a spectral set.

Whether a spectrum exists for a given measure $\mu$ is a fundamental question in harmonic analysis. Much of its significant interest stems from Fuglede's influential 1974 conjecture \cite{Fuglede01}:
$$\Omega \text{ is a spectral set} \Longleftrightarrow  \Omega \text{ tiles } \mathbb{R}^d \text{ by translations}.$$
The conjecture has remained open until 2004, at which point it was refuted in full generality for dimensions three and higher \cite{KM01,KM02,Tao}. Nonetheless, it continues to be open in a number of significant cases. For further progress in this direction, see \cite{L\'aba01,L\'aba02,Lagarias01,Lev01}. 

In this paper, for finite set $\mathcal{D}$, let us define the equally-weighted Dirac measure as
$$\delta_{\mathcal{D}}=\frac{1}{\#\mathcal{D}} \sum_{d\in\mathcal{D}} \delta_{d}.$$
The spectrality of the Dirac measure is closely related to the following concepts:
\begin{definition}
     Let $N\ge2$ be an integer, and let $\mathcal{D},L\subset\mathbb{Z}$ be finite sets with $\#\mathcal{D}=\#L$. We say that $(N,\mathcal{D},L)$ is a Hadamard triple if
$$\frac{1}{\sqrt{\#\mathcal{D}}}\left(e^{-2\pi i d\ell/N}\right)_{d\in \mathcal{D},\ell\in L}$$
is unitary. Equivalently, $L$ is a spectrum for the discrete measure $\delta_{N^{-1}\mathcal{D}}$.
\end{definition}
Purely atomic measures are not sufficient for a comprehensive understanding of all spectral measures. In 1998, Jorgensen and Pedersen \cite{JS01} discovered the first singular non-atomic spectral measure-the middle fourth Cantor measure. And surprisingly, for different spectra, the convergence results corresponding to their Fourier series are different, see \cite{Dutkay01,Strichartz01,Strichartz02}. Thus, attention turned to the problem of spectral eigenvalues- $t \in \mathbb{R}$ for which both $\Lambda$ and $t\Lambda$ are spectra of $\mu$. There has already been a large amount of research \cite{ADH02,Dai01,Fuhe01,FHW03,hetang01,LW03}. A set $\mathcal{A}\subset\mathbb{R}\setminus\{0\}$ is called a \textsl{spectral eigenvalue set} for $\mu$ if $\mu$ admits a spectrum $\Lambda$ satisfying that $a\Lambda$ is a spectrum of $\mu$ for every $a\in\mathcal{A}$. 
The notion of Beurling dimension is used to quantify the size of a discrete set. We recall its definition below:
$$\dim_{Be}(\Lambda)=\limsup_{h\rightarrow \infty} \sup_{x \in \mathbb{R}} \frac{\log \#(\Lambda \cap (x-\frac{h}{2},x+\frac{h}{2}))}{\log h}.$$
Shi \cite{Shi01} proved that, for an arbitrary spectral measure $\mu$, the Beurling dimension of any spectrum of $\mu$ is bounded above by the upper entropy dimension $\overline{\dim}_e\mu$. An and Lai \cite{AL04} constructed measures whose spectra have Beurling dimension zero. This naturally leads to the interpolation problem:

\begin{question}
    For $0 \leq s \leq \overline{\dim}_e \mu$, does there exist a spectrum $\Lambda$ of $\mu$ with $\dim_{Be}(\Lambda)=s?$
\end{question}

In 2019, Dutkay, Haussermann and Lai \cite{Dutkay02} proved that if $(N,\mathcal{D},L)$ forms a Hadamard triple, then the
associated self-affine measure $\mu_{N,\mathcal{D}}$ is a spectral measure. Subsequently, under the condition that $(N,\mathcal{D},L)$ is a Hadamard triple, Lu \cite{Luzheng01} proved that the set 
$$\mathcal{T}:=\{t\in\mathbb{Z}\setminus\{0\}:(N,\mathcal{D},tL)\text{ forms a Hadamard triple}\}$$
constitutes a spectral eigenvalue set for $\mu_{N,\mathcal{D}}$, and for any prescribed Beurling dimension $s \in [0,\frac{\log \#\mathcal{D}}{\log N}]$, there are continuum many corresponding spectra.

These developments naturally motivate the study of Moran measures generated by random convolutions of finite Hadamard triples. Let $(N_j,B_j,L_j)$ be finitely many one-dimensional Hadamard triples with $N_j\ge2$ for $1\le j\le m$.  Fix
$$\omega=(\omega_k)_{k\ge1}\in\Omega_m:=\{1,\dots,m\}^{\mathbb{N}},
 \qquad n_k\in\mathbb{N}^+.$$
Set
$$q_k=N_{\omega_k}^{n_k},\qquad
 D_k=B_{\omega_k},\qquad
 E_k=N_{\omega_k}^{n_k-1}L_{\omega_k},
 \qquad M_k=q_1q_2\cdots q_k.$$
Then $(q_k,D_k,E_k)$ is a Hadamard triple for every $k$.  We consider the random infinite convolution
\begin{equation}\label{eq:1.1}
 \mu=\delta_{M_1^{-1}D_1}*\delta_{M_2^{-1}D_2}*\cdots *\delta_{M_k^{-1}D_k}*\cdots .
\end{equation}
In \cite{LIMIAO01}, they proved that these measures are spectral under a suitable condition.

\begin{theorem}
    Let $(N_j,B_j,L_j)$ be finitely many Hadamard triples and $\gcd(B_j-B_j)=1$ for $1 \leq j \leq m$.  
    Let $\mu$ be given by \eqref{eq:1.1}, then $\mu$ is a spectral measure for every $\omega \in\Omega_m $.
\end{theorem}

The upper entropy dimension of such a measure is given by:
\begin{equation}\label{eq:1.2}
 \kappa_\omega=\limsup_{R\to\infty}\frac{\sum_{r=1}^R\log \#D_R}{\sum_{r=1}^R\log q_R}.
\end{equation}
 Define the common Hadamard triple multiplier set
\begin{equation}\label{eq:1.3}
 \mathcal{T}_*:=\bigcap_{j=1}^m\mathcal{T}_j,
 \qquad
 \mathcal{T}_j:=\{t\in\Z\setminus\{0\}:(N_j,B_j,tL_j)\text{ is a Hadamard triple}\}.
\end{equation}
The following are the main results of this paper.
\begin{theorem}\label{thm:1.3}
Let $(N_j,B_j,L_j)$ be finitely many Hadamard triples for all $1\le j\le m$. Assume the non-full-digit gap
\begin{equation}\label{eq:1.5}
 \rho:=\min_{1\le j\le m}\frac{N_j}{\#B_j}>1.
\end{equation}
Let $\mu$ be the random convolution in \eqref{eq:1.1}, and  $\kappa_\omega$ as in \eqref{eq:1.2}.  Then for every $s\in[0,\kappa_\omega]$ there are continuum many countable sets $\Lambda\subset\mathbb{Z}$ such that
$$
 t\Lambda\text{ is a spectrum of }\mu\quad\text{for every }t\in\mathcal{T}_*,
$$
and
$$
 \dim_{Be}\Lambda=s.
$$
Consequently, $\mathcal{T}_*$ is a spectral eigenvalue set for $\mu$. Let
$$
V_s(\mu,\mathcal T_*):=\{\Lambda\subset\mathbb Z: t\Lambda\text{ is a spectrum of }\mu\text{ for every }t\in\mathcal T_*,\ \dim_{Be}\Lambda=s\},
$$
then $\#V_s(\mu,\mathcal{T}_*)=2^{\aleph_0}$.
\end{theorem}

\begin{remark}
If all Hadamard triples are identical and $n_k\equiv1$, then Theorem \ref{thm:1.3} reduces to Lu's result for the self-similar measure. The set $\mathcal{T}_*$ contains all nonzero integers that are coprime to $N_1N_2\cdots N_m$, i.e.,
$$
 \{p\in\mathbb{Z}\setminus\{0\}:\gcd(p,N_1N_2\cdots N_m)=1\} \subset \mathcal{T}_*.
$$
\end{remark}

To clarify the difference between these two conditions, we give the following example.

\begin{example}
    For each \(j\in\{1,\ldots,m\}\), let \(N_j=4\), \(B_j=\{0,2\}\) and \(L_j=\{0,1\}\). It is easy to see that each $(N_j,B_j,L_j)$ is a  Hadamard triple. Moreover, $\gcd(B_j-B_j)=2$ and $\rho>1$.  In addition, $\mu_{4,\{0,2\}}$ is a spectral measure.
\end{example}

We organise our paper as follows: in section 2, we introduce some necessary preliminaries and notations, and present some lemmas that will be used in the proofs that follow. In Section 3, by establishing the Hadamard triple identity, we obtain an estimate of the tail term of the Fourier transform of the measure, preparing for the subsequent construction of the spectrum. In the last section, we construct a set satisfying the spectral property and prove Theorem \ref{thm:1.3}.

\section{Preliminaries}
In this section, we present some necessary definitions and fundamental facts that will be used in our research. 

The Fourier transform of a Borel probability measure $\mu$ is defined to be 
$$\widehat{\mu}(\xi)= \int e^{-2\pi i x \cdot \xi} d\mu(x).$$
We denote the convolution of the first $n$ terms
$$\mu_n=\delta_{M_1^{-1}D_1}*\cdots *\delta_{M_n^{-1}D_n}$$
and the tail term by
\[
\mu_{>n}
=
\delta_{M_{n+1}^{-1}D_{n+1}}
*
\delta_{M_{n+2}^{-1}D_{n+2}}
*
\cdots .
\]
According to \eqref{eq:1.1}, we have $$\mu=\mu_n*\mu_{>n}.$$ 
The scaled tail is denoted by
\begin{equation}\label{eq:2.1}
 \nu_{>n}(A)=\mu_{>n}(M_n^{-1}A),
 \qquad \text{ for Borel set } A\subset\mathbb{R}.
\end{equation}
Thus 
\begin{equation}\label{eq:2.2}
 \widehat{\mu}_{>n}(\xi)=\widehat{\nu}_{>n}(M_n^{-1}\xi).
\end{equation}
For $0\le a<b$, define the scale, the spatial digit block, and the frequency digit block by
$$
 Q_{a,b}=q_{a+1}q_{a+2}\cdots q_b,
$$
$$
 D_{a,b}=q_{a+2}\cdots q_bD_{a+1}+q_{a+3}\cdots q_bD_{a+2}+\cdots +q_bD_{b-1}+D_b,
$$
and
$$
 E_{a,b}=E_{a+1}+q_{a+1}E_{a+2}+\cdots +q_{a+1}q_{a+2}\cdots q_{b-1}E_b.
$$
Here and below, sums of sets are Minkowski sums, scalar multiplication of sets is understood elementwise, and empty products are taken to be $1$. We give some properties of Hadamard triples and omit its proof. 

\begin{lemma}[Basic operations for Hadamard triples]\label{lem:2.1}
Let $(N,B,L)$ be a Hadamard triple with $N\ge2$.
\begin{enumerate}[label=(\arabic*)]
\item $\delta_{N^{-1}B}$ is a spectral measure with spectrum $L$.

\item $(N,B,\widetilde L)$ is also a Hadamard triple for $\widetilde{L}=L\pmod{N}$.

\item For every $r\in\mathbb N^+$,
\[
 (N^r,B,N^{r-1}L)
\]
is a Hadamard triple. Indeed,
\[
 e^{-2\pi i b(N^{r-1}\ell)/N^r}=e^{-2\pi i b\ell/N}.
\]

\item More generally, if $(Q_i,D_i,E_i)$, $1\le i\le r$, are Hadamard triples, then
\[
\left(
\prod_{i=1}^r Q_i,
\sum_{i=1}^r\left(\prod_{\ell=i+1}^r Q_\ell\right)D_i,
\sum_{i=1}^r\left(\prod_{\ell=1}^{i-1}Q_\ell\right)E_i
\right)
\]
is a Hadamard triple.
\end{enumerate}
\end{lemma}

\begin{lemma}\label{lem:2.2}
For every $0\le a<b$, $(Q_{a,b},D_{a,b},E_{a,b})$ is a Hadamard triple.  Moreover, if $t\in\mathcal{T}_*$, then $(Q_{a,b},D_{a,b},tE_{a,b})$ is a Hadamard triple.
\end{lemma}

\begin{proof}
By Lemma \ref{lem:2.1} (3), we obtain that $$(q_k,D_k,E_k)=(N_{\omega_k}^{n_k},B_{\omega_k},N_{\omega_k}^{n_k-1}L_{\omega_k})$$ is a Hadamard triple. By Lemma \ref{lem:2.1} (4), the first assertion holds. If $t\in\mathcal{T}_*$, then $(N_j,B_j,tL_j)$ is a Hadamard triple for every $1\leq j \leq m$. By using Lemma \ref{lem:2.1} (3) and Lemma \ref{lem:2.1} (4) again, we obtain that $(Q_{a,b},D_{a,b},tE_{a,b})$ is a Hadamard triple.
\end{proof}

For $\eta=(\eta_k)\in\Omega_m$, define the unit-exponent random convolution
\[
 \mu_\eta=\delta_{N_{\eta_1}^{-1}B_{\eta_1}}*
 \delta_{(N_{\eta_1}N_{\eta_2})^{-1}B_{\eta_2}}*\cdots .
\]

\begin{lemma} \label{lem:2.3}
   The family $\{\widehat{\mu}_{\eta}:\eta \in \Omega_m\}$ is equicontinuous on $\mathbb{R}$, i.e., for each $\varepsilon>0$, there exists $\delta>0$, such that $\lvert\widehat{\mu}_{\eta}(x)-\widehat{\mu}_{\eta}(y)\rvert <\varepsilon$ for all $|x-y|<\delta$ and all $\eta\in \Omega_m$.
\end{lemma}
\begin{proof}
    By the equicontinuity argument in \cite[Lemma 2.3]{LIMIAO01}
and the proof of \cite[Corollary 4.4]{LIMIAO01}, we know that
\(\{\widehat{\mu}_{\eta}:\eta \in \Omega_m\}\) is equicontinuous.
\end{proof}

One may equip the symbolic space $\Omega_m$ with the metric 
$$d(\omega,\eta)=2^{-\min\{k \geq1:\omega_k \neq \eta_k \}}$$
whenever $\omega=(\omega_k), \eta=(\eta_k) $ are distinct elements of $\Omega_m$. Under this metric, $\Omega_m$ becomes a compact metric space. Moreover, a sequence $\{\omega^{(j)}\}_{j=1}^{\infty} \subset \Omega_m$ converges to $\omega$ if and only if for each $k \geq 1 $, there exists $j_0 \geq 1$ such that for every $j \geq j_0$, 
$$\omega_1^{(j)}\omega_2^{(j)}\cdots \omega_k^{(j)}=\omega_1\omega_2\cdots \omega_k.$$
\begin{lemma}\label{lem:2.4}
 Let $\xi \in \mathbb{R}$, the map
$$
\eta\mapsto \widehat{\mu}_{\eta}(\xi)
$$
is continuous on $\Omega_m$.

\end{lemma}
\begin{proof}

Let
$$
M=\max\{|b|: b\in B_j,\ 1\leq j\leq m\}.
$$
For $\eta\in\Omega_m$, write
$$
\mu_{\eta}
=
\mu_{\eta,n}*\mu_{\eta,>n},
$$
where
$$
\mu_{\eta,n}
=
\delta_{N_{\eta_1}^{-1}B_{\eta_1}}
*
\delta_{(N_{\eta_1}N_{\eta_2})^{-1}B_{\eta_2}}
*
\cdots
*
\delta_{(N_{\eta_1}\cdots N_{\eta_n})^{-1}B_{\eta_n}}.
$$
Since $|N_j|\geq 2$ for every $j$, the support of the tail measure
$\mu_{\eta,>n}$ is contained in an interval whose length tends to zero uniformly in
$\eta$. More precisely,
$$
\operatorname{spt}(\mu_{\eta,>n})\subset [-M2^{-n},M2^{-n}]
$$
for every $\eta\in\Omega_m$ and every $n\geq 1$. Let $I\subset\mathbb R$ be compact. Then, for $\xi\in I$,
\begin{align*}
   \left|\widehat{\mu}_{\eta}(\xi)-\widehat{\mu}_{\eta,n}(\xi)\right|
&=\left| \iint e^{-2\pi i\xi (x+y)} \,d\mu_{\eta,n}(x)\,d\mu_{\eta,>n}(y) -\int e^{-2\pi i\xi x}\,d\mu_{\eta,n}(x)\right|\\
&=\left| \iint e^{-2\pi i\xi (x+y)} \,d\mu_{\eta,n}(x)\,d\mu_{\eta,>n}(y) -\iint e^{-2\pi i\xi x}\,d\mu_{\eta,n}(x)\,d\mu_{\eta,>n}(y)\right|\\
&=
\left|
\iint e^{-2\pi i\xi x}
\left(e^{-2\pi i\xi y}-1\right)
\,d\mu_{\eta,n}(x)\,d\mu_{\eta,>n}(y)
\right|  \\
&\leq
\sup_{\xi\in I,\ |y|\leq M2^{-n}}
\left|e^{-2\pi i\xi y}-1\right|. 
\end{align*}
Since $I$ is compact, the right-hand side tends to $0$ as $n\to\infty$, uniformly in
$\eta\in\Omega_m$. Hence
$$
\sup_{\eta\in\Omega_m}\sup_{\xi\in I}
\left|\widehat{\mu}_{\eta}(\xi)-\widehat{\mu}_{\eta,n}(\xi)\right|
\rightarrow 0.
$$
Now suppose that $\eta^{(r)}\to\eta$ in $\Omega_m$. By the definition of the topology on
$\Omega_m$, for each fixed $n$, there exists $r_0$ such that, for all $r\geq r_0$,
$$
\eta^{(r)}_1\eta^{(r)}_2\cdots\eta^{(r)}_n
=
\eta_1\eta_2\cdots\eta_n.
$$
Therefore,
$$
\mu_{\eta^{(r)},n}=\mu_{\eta,n},
$$
and hence
$$
\widehat{\mu}_{\eta^{(r)},n}(\xi)
=
\widehat{\mu}_{\eta,n}(\xi)
\qquad \text{for all } \xi\in\mathbb R.
$$
Given $\varepsilon>0$, choose $n$ large enough so that
\[
\sup_{\eta\in\Omega_m}\sup_{\xi\in I}
\left|\widehat{\mu}_{\eta}(\xi)-\widehat{\mu}_{\eta,n}(\xi)\right|
<\frac{\varepsilon}{2}.
\]
For such $r \geq r_0$, we obtain
\begin{align*}
\sup_{\xi\in I}\left|\widehat{\mu}_{\eta^{(r)}}(\xi)-\widehat{\mu}_{\eta}(\xi)\right|\leq\sup_{\xi\in I}\left|\widehat{\mu}_{\eta^{(r)}}(\xi)-\widehat{\mu}_{\eta^{(r)},n}(\xi)\right| +\sup_{\xi\in I}
\left|\widehat{\mu}_{\eta,n}(\xi)-\widehat{\mu}_{\eta}(\xi)
\right|  <\varepsilon.\end{align*}
Thus,
$$
\sup_{\xi\in I}
\left|\widehat{\mu}_{\eta^{(r)}}(\xi)-\widehat{\mu}_{\eta}(\xi)\right|\to 0.$$
In particular, for every fixed $\xi\in\mathbb R$,
$$
\widehat{\mu}_{\eta^{(r)}}(\xi)\to \widehat{\mu}_{\eta}(\xi).
$$
Therefore, $\eta\mapsto \widehat{\mu}_{\eta}(\xi)$ is continuous for each fixed $\xi$, and the convergence is uniform for $\xi$ in compact subsets of $\mathbb R$.

\end{proof}

Finally, we conclude this section with a classical spectral criterion \cite{AHH01}.

\begin{theorem}\label{thm:2.5}
Suppose that \(\{\alpha_n\}_{n=1}^{\infty}\) is an increasing sequence of integers and \(\delta>0\).
If
\[
0\in\Lambda_{\alpha_n}\subset\Lambda_{\alpha_{n+1}},
\]
\(\Lambda_{\alpha_n}\) is a spectrum of \(\mu_{\alpha_n}\), and
\[
\inf_{n\ge1}
\inf_{\lambda \in \Lambda_{\alpha_n},\,|\xi|\le\delta}
\left|\widehat{\mu}_{>\alpha_n}(\lambda+\xi)\right|^2>0,
\]
then
\[
\Lambda=\bigcup_{n=1}^{\infty}\Lambda_{\alpha_n}
\]
is a spectrum of \(\mu\).
\end{theorem}

\section{Double lattice lemma}

This section contains many technical proofs, and the results obtained serve our subsequent constructions. In order to ultimately construct uncountably many spectra, we need to ensure that at each level there are at least two choices for the construction. This is the motivation for obtaining Corollary \ref{cor:3.3} and Lemma \ref{lem:3.5}.

\begin{lemma}[Hadamard triple counting identity]\label{lem:3.1}
Let $(Q,B,E)$ and $(Q,B,tE)$ be Hadamard triples for some $t\in\mathbb{Z}\setminus\{0\}$. Then
\begin{equation}\label{eq:3.1}
 \sum_{k=0}^{Q-1}\left|\widehat{\delta_{Q^{-1}B}}(\xi+tk)\right|^2=\frac{Q}{\#B},
 \qquad \xi\in\mathbb{R}.
\end{equation}
\end{lemma}

\begin{proof}
Since $(Q,B,tE)$ is a Hadamard triple, the matrix for $(Q,tB,E)$ has the same entries. Hence $(Q,tB,E)$ is also a Hadamard triple. Therefore, if $b\ne b'$ in $B$, we have
\begin{equation}\label{eq:3.2}
 Q\nmid t(b-b').
\end{equation}
Expanding the left side of \eqref{eq:3.1}, we obtain
\begin{align*}
\sum_{k=0}^{Q-1}\left|\widehat{\delta_{Q^{-1}B}}(\xi+tk)\right|^2=&\frac1{(\#B)^2}\sum_{k=0}^{Q-1}
 \sum_{b\in B}e^{-2\pi i b(\xi+tk)/Q}\sum_{b'\in B}e^{2\pi i b'(\xi+tk)/Q}
 \\
=&\frac1{(\#B)^2}\sum_{b\in B}\sum_{b'\in B}e^{-2\pi i(b-b')\xi/Q}
 \sum_{k=0}^{Q-1}e^{-2\pi i t(b-b')k/Q}.
\end{align*}
It is well known that
$$\sum_{k=0}^{Q-1}e^{-2\pi i t(b-b')k/Q}=\begin{cases}
Q, & \text{if } Q\mid t(b-b'),\\
0, & \text{if } Q\nmid t(b-b').
\end{cases}$$
By \eqref{eq:3.2}, if $Q\mid t(b-b')$, we have $b=b'$. Hence 
$$\sum_{k=0}^{Q-1}\left|\widehat{\delta_{Q^{-1}B}}(\xi+tk)\right|^2
=
\frac1{(\#B)^2}\sum_{b\in B}Q
=
\frac{Q}{\#B}.$$
\end{proof}

Li-Miao-Wang's finite-zero lemma \cite[Lemma 4.1]{LIMIAO01} says that for each $h>0$ the set
\begin{equation}\label{eq:3.3}
 [-h,h]\cap\bigcup_{\eta\in\Omega_m}\{\xi:\widehat{\mu_\eta}(\xi)=0\}
\end{equation}
is finite. Let $\sigma$ be the left shift operator on $\Omega_m$, i.e., $$\sigma(\omega_1\omega_2\cdots)=\omega_2\omega_3\cdots$$
where $(\omega_k) \in \Omega_m$. Here are the most important properties in this paper.

\begin{proposition}\label{prop:3.2}
Assume that \eqref{eq:1.5} holds.  Let $t\in\mathcal{T}_*\cap\mathbb{N}$.  Then there exist constants $\varepsilon_t>0$, $\delta_t>0$ and an integer $K_t\ge1$ such that for every $\eta\in\Omega_m$ and every $x\in[0,t]$ there are two distinct integers
$$
 k_1,k_2\in\{0,1,\dots,K_t\}
$$
for which
\begin{equation}\label{eq:3.4}
 |\widehat{\mu_\eta}(x+tk_i+y)|\ge\varepsilon_t,
 \qquad |y|<\delta_t,
 \quad i=1,2.
\end{equation}
\end{proposition}

\begin{proof}
Let
$$
 Z_t=[0,t]\cap\bigcup_{\eta\in\Omega_m}\{\xi:\widehat{\mu_\eta}(\xi)=0\}.
$$
By \eqref{eq:3.3}, $Z_t$ is finite. Choose $r\in \mathbb{N}^+$ so large that
$$\rho^r>\#Z_t+2.$$
For a word $u=u_1\cdots u_r$, define
$$
 Q_u=N_{u_1}\cdots N_{u_r},
$$
and
$$
 B_u=N_{u_2}\cdots N_{u_r}B_{u_1}+N_{u_3}\cdots N_{u_r}B_{u_2}+\cdots+B_{u_r}.
$$
Then
$$
 \#B_u=\prod_{i=1}^r\#B_{u_i},
 \qquad
 \frac{Q_u}{\#B_u}\ge\rho^r>\#Z_t+2.
$$
By Lemma \ref{lem:2.2}, both \((Q_u,B_u,E_u)\) and \((Q_u,B_u,tE_u)\) are Hadamard triples, where
\[
E_u=L_{u_1}+N_{u_1}L_{u_2}+\cdots+N_{u_1}\cdots N_{u_{r-1}}L_{u_r}.
\]
Lemma \ref{lem:3.1} gives, for all $\xi\in\R$,
\begin{equation*}
 \sum_{k=0}^{Q_u-1}\left|\widehat{\delta_{Q_u^{-1}B_u}}(\xi+tk)\right|^2
 =\frac{Q_u}{\#B_u}.
\end{equation*}
Since \(\delta_{Q_u^{-1}B_u}\) is a probability measure,
\[
\left|\widehat{\delta_{Q_u^{-1}B_u}}\right|\le1.
\]
Hence
\begin{equation}\label{eq:3.5}
\begin{split}
    &\#\left\{k \in \{0,1,\ldots ,Q_u-1\}: \widehat{\delta_{Q_u^{-1}B_u}}(\xi+tk) \neq 0\right\} \\ 
  = & \sum_{k=0}^{Q_u-1} \chi_{\{k:\widehat{\delta_{Q_u^{-1}B_u}}(\xi+tk) \neq 0\}}(k) \geq \sum_{k=0}^{Q_u-1}\left|\widehat{\delta_{Q_u^{-1}B_u}}(\xi+tk)\right|^2=\frac{Q_u}{\#B_u}.
\end{split}
\end{equation}
On the other hand, we note that $\frac{x+tk}{Q_u} \in [0,t]$ for all $0\leq k < Q_u$. This means that there are at most $\#Z_t$ values of $k$ such that 
\begin{equation}\label{eq:3.6}
\widehat{\mu_{\sigma^r\eta}}
\left(\frac{x+tk}{Q_u}\right)=0.
\end{equation}
Fix $\eta\in\Omega_m$ and $x\in[0,t]$, and let $u=\eta_1\cdots\eta_r$.  For $0\le k<Q_u$,
$$
 \widehat{\mu_\eta}(x+tk)=
 \widehat{\delta_{Q_u^{-1}B_u}}(x+tk)
 \widehat{\mu_{\sigma^r\eta}}\left(\frac{x+tk}{Q_u}\right).
$$
Combining \eqref{eq:3.5}, \eqref{eq:3.6} and $\frac{Q_u}{\#B_u}>\#Z_t+2$, we see that at least two numbers $k_1$ and $k_2$ can be chosen from the set $\{0,1,\ldots ,Q_u-1\}$ such that 
$$\widehat{\mu_\eta}(x+tk_i)\neq 0 \qquad \text{ for }i=1,2.$$
Let
\[
 K_t=\max\{Q_u:|u|=r\}-1.
\]
For each $0\le j\le K_t$, define
\[
F_j(\eta,x)=|\widehat{\mu_\eta}(x+tj)|,
\qquad (\eta,x)\in\Omega_m\times[0,t].
\]
Also define
\[
G(\eta,x)=\max_{0\le i<j\le K_t}\min\{F_i(\eta,x),F_j(\eta,x)\}.
\]
The preceding counting argument shows that $G(\eta,x)>0$ for every
$(\eta,x)\in\Omega_m\times[0,t]$. By Lemma \ref{lem:2.4} and the continuity of Fourier transforms in the frequency variable, $G$ is continuous on the compact set $\Omega_m\times[0,t]$. Therefore
\[
\varepsilon_t^{(0)}:=\min_{(\eta,x)\in\Omega_m\times[0,t]}G(\eta,x)>0.
\]
For every $(\eta,x)$, choose two distinct integers
$k_1,k_2\in\{0,1,\ldots,K_t\}$ such that
\[
|\widehat{\mu_\eta}(x+tk_i)|\ge \varepsilon_t^{(0)},
\qquad i=1,2.
\]
By Lemma \ref{lem:2.3}, the family
$\{\widehat{\mu_\eta}:\eta\in\Omega_m\}$ is equicontinuous on compact intervals. Hence there exists $\delta_t>0$ such that, whenever $|y|<\delta_t$,
\[
|\widehat{\mu_\eta}(x+tk_i+y)-\widehat{\mu_\eta}(x+tk_i)|
<\frac12\varepsilon_t^{(0)}
\]
for all $\eta\in\Omega_m$, all $x\in[0,t]$, all $0\le k_i\le K_t$, and $i=1,2$. Taking
\[
\varepsilon_t=\frac12\varepsilon_t^{(0)}
\]
gives \eqref{eq:3.4}.

\end{proof}

We can obtain the connection between $\mu_{\eta}$ and the tails of \eqref{eq:1.1} by completing the missing terms. For any $\omega=(\omega_k) \in \Omega_m$ and $\{n_k\} \subset \mathbb{N}^+$, we expand the sequence $\omega$ by repeating each $\omega_k$ exactly $n_k$ times:
 $$
 \eta=\underbrace{\omega_1\cdots\omega_1}_{n_1}
 \underbrace{\omega_2\cdots\omega_2}_{n_2}
 \underbrace{\omega_3\cdots\omega_3}_{n_3}\cdots .
$$
Write $s_0=0$ and $s_k=n_1+\cdots+n_k$ for $k\ge1$. The standard factorization gives
\begin{equation*}
 \mu_{\sigma^{s_k}\eta}=\nu_{>k}*\rho_k
\end{equation*}
for some probability measure $\rho_k$.  Consequently,
\begin{equation}\label{eq:3.7}
 |\widehat{\mu_{\sigma^{s_k}\eta}}(\xi)|\le |\widehat{\nu}_{>k}(\xi)|.
\end{equation}
\begin{corollary}\label{cor:3.3}
Let $t\in\mathcal{T}_*\cap\mathbb{N}$.  There exist $\varepsilon_t>0$, $\delta_t>0$ and $K_t\ge1$ such that for every tail index $k\ge0$ and every $x\in[0,t]$ there are two distinct integers
 $$
 r_1,r_2\in\{0,1,\dots,K_t\}
$$
satisfying
 $$
 |\widehat{\nu}_{>k}(x+tr_i+y)|\ge\varepsilon_t,
 \qquad |y|<\delta_t,
 \quad i=1,2.
$$
\end{corollary}

\begin{proof}
Apply Proposition \ref{prop:3.2} to $\sigma^{s_k}\eta$ and use \eqref{eq:3.7}.
\end{proof}

\begin{remark}
    Note that the above conclusion and the following conclusion hold for an arbitrary $\omega \in \Omega_m$ and $\{n_k\} \subset \mathbb{N}^{+}$.
\end{remark}

In addition to ensuring that the constructed set satisfies the spectral property, we also need to ensure that a part of it is very sparse. The following lemma guarantees that this can be achieved.

\begin{lemma}\label{lem:3.5}
Let \(t\in\mathcal{T}_*\cap\mathbb N\). There exist constants
\(\epsilon_t>0\) and \(\delta_t>0\) such that, for every tail index \(k\),
every \(x\in[0,t]\), every \(R>0\), and every finite set
\(F\subset\mathbb Z\), one can find two distinct integers
\(r_1,r_2\notin F\) with \(|r_i|>R\) such that
\[
|\widehat{\nu}_{>k}(x+tr_i+y)|\ge \epsilon_t,
\qquad |y|<\delta_t,\quad i=1,2.
\]
\end{lemma}
\begin{proof}
By Corollary \ref{cor:3.3}, choose two distinct integers
\[
a_1,a_2\in\{0,1,\ldots,K_t\}
\]
such that
\[
 |\widehat{\nu}_{>k}(x+ta_j+y)|\ge\varepsilon_t,
 \qquad |y|<\delta_t,
 \quad j=1,2.
\]
Set
\[
Q_\ell=N_{\omega_{k+1}}^{n_{k+1}}\cdots N_{\omega_{k+\ell}}^{n_{k+\ell}}.
\]
Choose $\ell$ so large that
\[
Q_\ell>K_t+1
\]
and
\[
Q_\ell>\max\{R,\max_{u\in F}|u|\}+K_t+1.
\]
Decompose
\[
\widehat{\nu}_{>k}(\xi)=\widehat{\nu}_{k,k+\ell}(\xi)\widehat{\nu}_{>k+\ell}(Q_\ell^{-1}\xi).
\]
Since $|\widehat{\nu}_{>k+\ell}|\le1$, we have
\[
|\widehat{\nu}_{k,k+\ell}(x+ta_j+y)|\ge\varepsilon_t,
\qquad |y|<\delta_t,\quad j=1,2.
\]
Let
\[
x_j=Q_\ell^{-1}(x+ta_j),\qquad j=1,2.
\]
Since $x\in[0,t]$ and $0\le a_j\le K_t$, the choice $Q_\ell>K_t+1$ gives
\[
0\le x_j=Q_\ell^{-1}(x+ta_j)\le Q_\ell^{-1}t(1+K_t)<t.
\]
Hence $x_j\in[0,t]$ for $j=1,2$.
Applying Corollary \ref{cor:3.3} to the tail index $k+\ell$ and to $x_j$, we choose a nonzero integer $h_j\in\{1,\ldots,K_t\}$ such that
\[
|\widehat{\nu}_{>k+\ell}(x_j+th_j+Q_\ell^{-1}y)|\ge\varepsilon_t,
\qquad |y|<\delta_t,\quad j=1,2.
\]
Define
\[
r_j=a_j+Q_\ell h_j,
\qquad j=1,2.
\]
Then $r_j\notin F$ and $|r_j|>R$ by the choice of $Q_\ell$. Moreover, $r_1\ne r_2$: if not, then
\[
Q_\ell(h_1-h_2)=a_2-a_1,
\]
which is impossible because $Q_\ell>K_t\ge |a_2-a_1|$, unless $h_1=h_2$ and then $a_1=a_2$, a contradiction. The measure $\nu_{k,k+\ell}$ is supported on $Q_\ell^{-1}\mathbb Z$, and hence $\widehat{\nu}_{k,k+\ell}$ is $Q_\ell$-periodic. Thus
\[
|\widehat{\nu}_{k,k+\ell}(x+tr_j+y)|
=
|\widehat{\nu}_{k,k+\ell}(x+ta_j+y)|
\ge\varepsilon_t.
\]
Also,
\[
Q_\ell^{-1}(x+tr_j+y)=x_j+th_j+Q_\ell^{-1}y.
\]
Therefore
\[
 |\widehat{\nu}_{>k}(x+tr_j+y)|\ge\varepsilon_t^2,
 \qquad |y|<\delta_t,
 \quad j=1,2.
\]
Taking $\epsilon_t=\varepsilon_t^2$ proves the lemma.
\end{proof}

\section{Construction of the set}
In this section, our goal is to construct a set such that $\mathcal{T}_*$ becomes a spectral eigenvalue set of $\mu$. Since each level of the measure $\mu$ possesses the Hadamard triple property, a natural construction is as follows:
$$\Lambda=\bigcup_{n=1}^{\infty}E_{0,n}, \qquad E_{0,n}=E_1+q_1E_2+\cdots +q_1q_2 \cdots q_{n-1}E_n.$$
However, such a construction is too restrictive. Therefore, based on the favorable property that it remains a Hadamard triple after taking modulo, we adjust the modulo operation on certain layers, giving us more choices in constructing the set.

Let
$$\mathcal{T}_*\cap\mathbb{N}=\{p_1,p_2,\dots\}.$$
and choose a traversal sequence
\begin{equation}\label{eq:4.1}
 a_1,a_2,a_3,\dots=p_1,p_1,p_2,p_1,p_2,p_3,\dots,
\end{equation}
so that each $p_i$ occurs infinitely many times.\\

\textbf{Step 1: Construct the first layer.} Write $n_0=0$.  Taking $n_1 \in \mathbb{N}^{+}$ such that $\left\lvert M_{n_1}^{-1}a_1 \xi \right\rvert \leq \delta_{a_1}$ for all $|\xi|\le 1/3$. Let $C_1\subset E_{n_0,n_1}$ be any finite set containing $0$. We take $E_{n_0,n_1}$ as representatives modulo $M_{n_1}$. Define the set
$$\overline{\Lambda}_{n_1}=\{c_1+M_{n_1}k(c_1,n_1):c_1 \in C_1\},$$
where $k(c_1,n_1)$ is the integer obtained by applying Corollary \ref{cor:3.3} to point $M_{n_1}^{-1}a_1c_1$ for $c_1 \in C_1$. And define the set 
$$\widetilde{\Lambda}_{n_1}=\{e_1+M_{n_1}k(e_1,n_1):e_1 \in E_{n_0,n_1} \backslash C_1\},$$
where $k(e_1,n_1)$ is the integer obtained by applying Lemma \ref{lem:3.5} to point $M_{n_1}^{-1}a_1e_1$ for $e_1 \in E_{n_0,n_1} \backslash C_1$. And note that, according to Lemma \ref{lem:3.5}, the sparse elements can be chosen successively so that, after enumeration as $\sigma_1,\sigma_2,\ldots$, they satisfy
\[
|\sigma_{N+1}|>2|\sigma_N|+1.
\]

\textbf{Step 2: Construct the general layer.} Assume that $\overline{\Lambda}_{n_1}$ and $\widetilde{\Lambda}_{n_1}$ of the $n_{r-1}$-th layer have already been constructed. we now proceed to construct the $n_{r}$-th layer. Taking $n_{r} \in \mathbb{N}^{+}$ such that 
$$\left\lvert M_{n_r}^{-1}(a_{r} \lambda+\xi) \right\rvert \leq \delta_{a_{r}} \text{ for any } \lambda \in \Lambda_{n_{r-1}}:= \overline{\Lambda}_{n_{r-1}} \cup \widetilde{\Lambda}_{n_{r-1}} \text{ and } |\xi|\leq 1/3.$$
Let $C_r\subset E_{n_{r-1},n_r}$ be any finite set containing $0$.
We take \(E_{n_{r-1},n_r}\) as representatives modulo
\(q_{n_{r-1}+1}\cdots q_{n_r}\). Define the set
$$\overline{\Lambda}_{n_r}=\overline{\Lambda}_{n_{r-1}}+M_{n_{r-1}}\{c_r+q_{n_{r-1}+1}q_{n_{r-1}+2}\cdots q_{n_r}k(c_r,n_r):c_r \in C_r\}$$
where $k(c_r,n_r)$ is the integer obtained by applying Corollary \ref{cor:3.3} to point $M_{n_{r-1}}M_{n_r}^{-1}a_rc_r$ for $c_r \in C_r$. And define the set
\begin{align*}
\widetilde{\Lambda}_{n_r}=&\left(\overline{\Lambda}_{n_{r-1}}+M_{n_{r-1}}\{e_r+q_{n_{r-1}+1}q_{n_{r-1}+2}\cdots q_{n_r}k(e_r,n_r):e_r \in E_{n_{r-1},n_r} \backslash C_r\}\right)\\
    \cup & \left(\widetilde{\Lambda}_{n_{r-1}}+M_{n_{r-1}}\{e_r'+q_{n_{r-1}+1}q_{n_{r-1}+2}\cdots q_{n_r}k(e_r',n_r):e_r' \in E_{n_{r-1},n_r} \}\right)
\end{align*}
where \(k(e_r,n_r)\) is the integer obtained by applying Lemma \ref{lem:3.5}
to point \(M_{n_{r-1}}M_{n_r}^{-1}a_re_r\) for
\(e_r\in E_{n_{r-1},n_r}\backslash C_r\), and similarly to point
\(M_{n_{r-1}}M_{n_r}^{-1}a_re_r'\) for
\(e_r'\in E_{n_{r-1},n_r}\). And note that, according to Lemma \ref{lem:3.5}, the new sparse elements are chosen so that the enumeration $\sigma_1,\sigma_2,\ldots$ continues to satisfy
\[
|\sigma_{N+1}|>2|\sigma_N|+1.
\]

After such an inductive construction, we always choose the zero digit with carry
\[
k(0,n_r)=0.
\]
Hence
\[
\overline{\Lambda}_{n_{r-1}}\subset \overline{\Lambda}_{n_r}
\text{ and }
\widetilde{\Lambda}_{n_{r-1}}\subset\widetilde{\Lambda}_{n_r}.
\]
Define
$$\Lambda_{n_r}=\overline{\Lambda}_{n_r} \cup\widetilde{\Lambda}_{n_r},$$
we know that $\Lambda_{n_{r-1}}\subset\Lambda_{n_r}$ for any $r \in \mathbb{N}^+$. Write $\Lambda=\bigcup_{r=1}^{\infty} \Lambda_{n_r}$.\\

\textbf{Step 3: $ t\Lambda_{n_r}$ is a spectrum of $\mu_{n_r}$ for every $t\in\mathcal{T}_*$.}
According to the construction of the first layer, for any two distinct $\lambda, \lambda' \in t\Lambda_{n_1}$, we find that
\begin{equation}\label{eq:4.2}
    \lambda-\lambda' \in tE_{n_0,n_1}-tE_{n_0,n_1}+tQ_{n_0,n_1} \mathbb{Z}.
\end{equation}
Moreover, we know that
\begin{equation}\label{eq:4.3}
\widehat{\mu}_{n_1}=\widehat{\delta_{M_1^{-1}D_1}}\widehat{\delta_{M_2^{-1}D_2}}\cdots \widehat{\delta_{M_{n_1}^{-1}D_{n_1}}}=\widehat{\delta_{Q_{n_0,n_1}^{-1}D_{n_0,n_1}}}
\end{equation}
Combining \eqref{eq:4.2}, \eqref{eq:4.3}, Lemma \ref{lem:2.1} and 
Lemma \ref{lem:2.2}, we obtain that $ t\Lambda_{n_1}$ is a spectrum of $\mu_{n_1}$ for every $t\in\mathcal{T}_*$. Assuming that this holds at $r-1$-th layer; we now consider the 
$r$-th layer. Let  $t\lambda_r, t\lambda_r' \in t\Lambda_{n_r}$ be distinct. We distinguish three cases.\\

(1) $\lambda_r, \lambda_r' \in \overline{\Lambda}_{n_r}$. Then there exist two elements $\lambda_{r-1}, \lambda_{r-1}' \in \overline{\Lambda}_{n_{r-1}}$ such that
$$t\lambda_r=t\lambda_{r-1}+tM_{n_{r-1}}c_r+tM_{n_{r}}k(c_r,n_r)$$
and 
$$t\lambda_r'=t\lambda_{r-1}'+tM_{n_{r-1}}c_r'+tM_{n_{r}}k(c_r',n_r),$$
where $c_r,c_r' \in E_{n_{r-1},n_{r}}$. If $\lambda_{r-1}\neq \lambda_{r-1}'$, then we have
$$t\lambda_r-t\lambda_r'\in t\lambda_{r-1}-t\lambda_{r-1}'+M_{n_{r-1}}\mathbb{Z} \subset \mathcal{Z}(\widehat{\mu_{n_{r-1}}}) \subset \mathcal{Z}(\widehat{\mu_{n_{r}}}).$$
If $\lambda_{r-1}= \lambda_{r-1}'$, then $c_r \neq c_r'$ and 
$$t\lambda_r-t\lambda_r'\in M_{n_{r-1}}(tc_r-tc_r')+M_{n_{r}}\mathbb{Z} \subset M_{n_{r-1}}(tE_{n_{r-1},n_{r}}-tE_{n_{r-1},n_{r}}+Q_{n_{r-1},n_{r}}\mathbb{Z})$$
According to Lemma \ref{lem:2.1} and 
Lemma \ref{lem:2.2}, we obtain that
\[
t\lambda_r-t\lambda_r'
\in
M_{n_{r-1}}\mathcal Z
\left(
\widehat{\delta_{Q_{n_{r-1},n_r}^{-1}D_{n_{r-1},n_r}}}
\right)
\subset
\mathcal Z(\widehat{\mu_{n_r}}).
\]

(2) $\lambda_r \in \overline{\Lambda}_{n_r}$, $\lambda_r' \in \widetilde{\Lambda}_{n_r}$. Then there exist two elements $\lambda_{r-1} \in \overline{\Lambda}_{n_{r-1}}$ and $c_r\in C_r$ such that
$$t\lambda_r=t\lambda_{r-1}+tM_{n_{r-1}}c_r+tM_{n_{r}}k(c_r,n_r).$$
And there exist two elements $\lambda_{r-1}' \in \overline{\Lambda}_{n_{r-1}}$, $e_r\in E_{n_{r-1},n_r} \backslash C_r$ or $\lambda_{r-1}' \in \widetilde{\Lambda}_{n_{r-1}}$, $e_r\in E_{n_{r-1},n_r} $ such that
$$t\lambda_r'=t\lambda_{r-1}'+tM_{n_{r-1}}e_r+tM_{n_{r}}k(e_r,n_r).$$
Due to $\lambda_r\neq \lambda_r'$, we sure that $(\lambda_{r-1},c_r) \neq (\lambda_{r-1}',e_r)$. If $\lambda_{r-1}\neq \lambda_{r-1}'$, then 
$$t\lambda_r-t\lambda_r'\in t\lambda_{r-1}-t\lambda_{r-1}'+M_{n_{r-1}}\mathbb{Z} \subset \mathcal{Z}(\widehat{\mu_{n_{r-1}}}) \subset \mathcal{Z}(\widehat{\mu_{n_{r}}}).$$
If \(\lambda_{r-1}= \lambda_{r-1}'\), then \(c_r \neq e_r\) and
$$t\lambda_r-t\lambda_r'\in M_{n_{r-1}}(tc_r-te_r)+M_{n_{r}}\mathbb{Z} \subset M_{n_{r-1}}(tE_{n_{r-1},n_{r}}-tE_{n_{r-1},n_{r}}+Q_{n_{r-1},n_{r}}\mathbb{Z})$$
Again according to Lemma \ref{lem:2.1} and 
Lemma \ref{lem:2.2}, we obtain that
$$t\lambda_r-t\lambda_r'\in M_{n_{r-1}}\mathcal{Z}(\widehat{\delta_{Q_{n_{r-1},n_{r}}^{-1}D_{n_{r-1},n_{r}}}}) =\mathcal{Z}(\widehat{\mu_{n_{r}}}).$$

(3) $\lambda_r, \lambda_r' \in \widetilde{\Lambda}_{n_r}$. If the parent frequencies are distinct, the induction hypothesis applies. If the parent frequencies coincide, then the block digits are distinct and the block Hadamard triple in Lemma \ref{lem:2.2} applies. Thus this case is handled in the same way as the previous two cases.\\

The above three cases all show that \(t\Lambda_{n_r}\) is an orthogonal set for the measure \(\mu_{n_r}\). Since $ t\Lambda_{n_{r-1}}$ is a spectrum of $\mu_{n_{r-1}}$, then 
$$\# \Lambda_{n_{r-1}}=\#\supp(\mu_{n_{r-1}})=\#D_1\#D_2 \cdots \#D_{n_{r-1}}.$$
The map
\[
(\lambda,e)\mapsto \lambda+M_{n_{r-1}}e+M_{n_r}k(e,n_r)
\]
is injective. Indeed, if two such values coincide, reducing modulo
\(M_{n_{r-1}}\) gives congruent parent frequencies modulo \(M_{n_{r-1}}\). By the induction hypothesis, the finite spectrum has distinct residues modulo \(M_{n_{r-1}}\), so the parent frequencies are equal. After dividing by
\(M_{n_{r-1}}\) and reducing modulo
\(q_{n_{r-1}+1}\cdots q_{n_r}\), the corresponding block digits are equal.
Furthermore, by construction, we see that
$$\#\Lambda_{n_r}=\#\Lambda_{n_{r-1}}\#E_{n_{r-1},n_r}=\#D_1\#D_2 \cdots \#D_{n_{r}}=\#\supp(\mu_{n_{r}}).$$
Hence we can claim that $t\Lambda_{n_r}$ is a spectrum of $\mu_{n_{r}}$.\\

\textbf{Step 4: For any $r\in \mathbb{N}$, $\lambda_r \in \Lambda_{n_r}$ and $|y|<\delta_{a_r}$,  $|\widehat{\nu}_{>{n_r}}(M_{n_r}^{-1}(a_r\lambda_r+y))|\ge \varepsilon_{a_r}$.} For every $\lambda_1 \in \Lambda_{n_1}$. If $\lambda_1\in \overline{\Lambda}_{n_1}$, then $$\lambda_1=c_1+M_{n_1}k(c_1,n_1).$$ 
The selection rule of $k(c_1,n_1)$ together with Corollary  \ref{cor:3.3} yields that for $|y|<\delta_{a_1}$,
$$|\widehat{\nu}_{>{n_1}}(M_{n_1}^{-1}(a_1\lambda_1+y))|=|\widehat{\nu}_{>{n_1}}(M_{n_1}^{-1}a_1c_1+a_1k(c_1,n_1)+M_{n_1}^{-1}y)|\ge\varepsilon_{a_1}.$$
If $\lambda_1\in \widetilde{\Lambda}_{n_1}$, then $$\lambda_1=e_1+M_{n_1}k(e_1,n_1).$$ 
From the selection rule of $k(e_1,n_1)$ and Lemma \ref{lem:3.5}, it follows that  for $|y|<\delta_{a_1}$, 
$$|\widehat{\nu}_{>{n_1}}(M_{n_1}^{-1}(a_1\lambda_1+y))|=|\widehat{\nu}_{>{n_1}}(M_{n_1}^{-1}a_1e_1+a_1k(e_1,n_1)+M_{n_1}^{-1}y)|\ge\epsilon_{a_1} .$$
Assuming that this holds at $r-1$-th layer; we now consider the 
$r$-th layer. For any $\lambda_r \in \Lambda_{n_r}$, if \(\lambda_r\in\overline{\Lambda}_{n_r}\), then there exists
\(\lambda_{r-1}\in\overline{\Lambda}_{n_{r-1}}\) and \(c_r\in C_r\) such that
$$\lambda_r=\lambda_{r-1}+M_{n_{r-1}}c_r+M_{n_{r}}k(c_r,n_r).$$
According to $\left\lvert M_{n_r}^{-1}(a_{r} \lambda_{r-1} +y)\right\rvert \leq \delta_{a_{r}}$, the selection rule of $k(c_r,n_r)$ and Corollary \ref{cor:3.3}, we can obtain that
$$|\widehat{\nu}_{>{n_r}}(M_{n_r}^{-1}(a_r\lambda_r+y))|=|\widehat{\nu}_{>{n_r}}(M_{n_r}^{-1}(a_r\lambda_{r-1}+y)+M_{n_{r-1}}M_{n_r}^{-1}c_ra_r+a_rk(c_r,n_r))| \geq \varepsilon_{a_r}.$$
If $\lambda_r \in \widetilde{\Lambda}_{n_r}$, there exist two elements $\lambda_{r-1} \in \overline{\Lambda}_{n_{r-1}}$, $e_r\in E_{n_{r-1},n_r} \backslash C_r$ or $\lambda_{r-1} \in \widetilde{\Lambda}_{n_{r-1}}$, $e_r\in E_{n_{r-1},n_r} $ such that
$$\lambda_r=\lambda_{r-1}+M_{n_{r-1}}e_r+M_{n_{r}}k(e_r,n_r).$$
According to $\left\lvert M_{n_r}^{-1}(a_{r} \lambda_{r-1} +y)\right\rvert \leq \delta_{a_{r}}$, the selection rule of $k(e_r,n_r)$ and Lemma \ref{lem:3.5}, we deduce that
$$|\widehat{\nu}_{>{n_r}}(M_{n_r}^{-1}(a_r\lambda_r+y))|=|\widehat{\nu}_{>{n_r}}(M_{n_r}^{-1}(a_r\lambda_{r-1}+y)+M_{n_{r-1}}M_{n_r}^{-1}e_ra_r+a_rk(e_r,n_r))| \geq \varepsilon_{a_r}.$$\\

\begin{proposition}\label{prop: 4.1}
For every $t\in\mathcal{T}_*$, the set $t\Lambda$ is a spectrum of $\mu$.   
\end{proposition}
\begin{proof}
    It suffices to treat $t\in\mathcal{T}_*\cap\mathbb{N}$, since negative multipliers follow by complex conjugation.
Fix $t\in\mathcal{T}_*\cap\mathbb{N}$. By the traversal property \eqref{eq:4.1}, there exists an increasing subsequence $\{n_{m_j}\}$ such that $a_{m_j}=t$.  For each $j$, by \textbf{step 3}, $t\Lambda_{n_{m_j}}$ is a spectrum of the finite convolution $\mu_{n_{m_j}}$. 

By \eqref{eq:2.2} and \textbf{step 4}, we know that there exist constants $\varepsilon_{t}>0$ and $\delta_{t}>0$ such that for any $j\in \mathbb{N}$,
$$|\widehat{\mu}_{>{n_{m_j}}}(t\lambda+y)|=|\widehat{\nu}_{>{n_{m_j}}}(M_{n_{m_j}}^{-1}(t\lambda+y))| \ge \varepsilon_{t} $$
for $\lambda \in \Lambda_{n_{m_j}}$ and $|y|<\delta_{t}$. Then by Theorem \ref{thm:2.5}, $t\Lambda$ is  a spectrum of $\mu$.
\end{proof}

We choose the block sizes and the principal-cardinality sequence.

\begin{lemma}\label{lem:4.2}
Assume that $\kappa_\omega$ as in \eqref{eq:1.2}.  For every $s\in[0,\kappa_\omega]$ there are integers
\[
 0=m_0<m_1<m_2<\cdots
\]
and $0\in C_r\subset E_{m_{r-1},m_r}$ such that
\begin{equation}\label{eq:4.4}
 \limsup_{R\to\infty}\frac{\sum_{r=1}^R\log \#C_r}{\log M_{m_R}}=s.
\end{equation}
\end{lemma}

\begin{proof}
We consider the following three cases: $s=0$, $s=\kappa_\omega$ and $0<s <\kappa_\omega$.

(i) $s=0$. Taking 
$C_r=\{0\}$, then $\sum_{r=1}^R\log \#C_r=0$.

(ii) $s=\kappa_\omega$. Taking 
$C_r=E_{m_{r-1},m_r}$, then
\[
\sum_{r=1}^R\log \#C_r
=
\sum_{r=1}^R\log \#E_{m_{r-1},m_r}
=
\log \#E_{0,m_R}
=
\sum_{k=1}^{m_R}\log \#D_k.
\]

(iii) $0<s <\kappa_\omega$. When $m_{r-1}$ has been fixed, we want to make $m_{r}$ sufficiently large. This allows the choice of $C_r$ to become the dominant term among the first $r$ terms, so that we can ignore the influence of the preceding $r-1$ terms. This implies that we want to make
$$\#C_R \asymp (q_{m_{R-1}+1}\cdots q_{m_{R}})^s.$$
Since
\[
s<\limsup_{N\to\infty}\frac{\sum_{k=1}^{N}\log \#D_k}{\sum_{k=1}^{N}\log q_k},
\]
and the initial segment up to \(m_{r-1}\) is fixed, we can choose \(m_r\) arbitrarily large such that
\[
s< \frac{\sum_{k=m_{r-1}+1}^{m_{r}}\log \#D_k}{\sum_{k=m_{r-1}+1}^{m_{r}}\log q_k}.
\]
We choose \(m_r\) further so that
\[
2^{-s(m_r-m_{r-1})}\le 2^{-r}.
\]
Then
\[
\prod_{k=m_{r-1}+1}^{m_{r}} q_k^s
<
\prod_{k=m_{r-1}+1}^{m_{r}}\#D_k
= \# E_{m_{r-1},m_{r}}.
\]
Hence we can choose \(C_r \subset E_{m_{r-1},m_{r}}\) satisfying
\[
\#C_r= \left\lfloor\prod_{k=m_{r-1}+1}^{m_{r}} q_k^s\right\rfloor,
\]
and
\[
\log \#C_r \leq s\sum_{k=m_{r-1}+1}^{m_{r}} \log q_k \leq \log (\#C_r+1).
\]
It follows that
\[
\left|\log \#C_r-s\sum_{k=m_{r-1}+1}^{m_{r}} \log q_k  \right|
\leq \log (\#C_r+1)- \log \#C_r
\leq \frac{1}{\# C_r}
\leq 2^{-s(m_r-m_{r-1})}
\leq 2^{-r}.
\]
Therefore the error sum is bounded, and hence
\[
\lim_{R\to\infty}\frac{\sum_{r=1}^R\log \#C_r}{\log M_{m_R}}=s.
\]

\end{proof}

\begin{remark}
    Note that the sequence $n_k$ chosen in the construction of the preliminary spectrum and the sequence $m_k$ chosen in Lemma \ref{lem:4.2} can be taken to be the same; this can be seen from the proof.
\end{remark}

Set 
$$\overline{\Lambda}=\bigcup_{r=0}^{\infty}\overline{\Lambda}_{n_r}, \qquad \widetilde{\Lambda}=\Lambda \backslash \overline{\Lambda}.$$
According to our previous construction rules, the elements of \(\widetilde\Lambda\) can be enumerated as
\(\sigma_1,\sigma_2,\ldots\) so that
\[
|\sigma_{N+1}|>2|\sigma_N|+1.
\]
Hence \(\dim_{Be}\widetilde{\Lambda}=0\). Write 
$$S(t)=\{n_{r}: a_{r}=t\}:=\{n_{j_1},n_{j_2},\ldots\} \text{ for some } t\in\mathcal{T}_*.$$
Since the sequence of sets $\overline{\Lambda}_{n_{j_i}}$ is monotonically increasing, we have  $\overline{\Lambda}=\bigcup_{n_{j_i} \in S(t)}\overline{\Lambda}_{n_{j_i}}$.

\begin{proposition}\label{prop:4.4}
The principal set satisfies $\dim_{Be}\overline{\Lambda}=s$.  Consequently, $\dim_{Be}\Lambda=s$.
\end{proposition}
\begin{proof}
    The levels \(n_{j_i}\) are chosen large enough so that
\[
M_{n_{j_i}}^{-1}|\lambda'|\le1
\]
for all \(\lambda'\in\overline{\Lambda}_{n_{j_i-1}}\). For any $\lambda \in \overline{\Lambda}_{n_{j_i}}$, there exist $\lambda' \in \overline{\Lambda}_{n_{j_i-1}}$ and $c_{j_i} \in C_{j_i}\subset E_{n_{j_i-1},n_{j_i}}$ such that 
\begin{align*}
    |\lambda|=&|\lambda'+M_{n_{j_i-1}}c_{j_i}+M_{n_{j_i}}k(c_{j_i},n_{j_i})|\\
    \leq &M_{n_{j_i}}|M_{n_{j_i}}^{-1}\lambda'+M_{n_{j_i}}^{-1}M_{n_{j_i-1}}c_{j_i}+k(c_{j_i},n_{j_i})|\leq 3M_{n_{j_i}}K_t.
\end{align*}
Note that 
$$\#\overline{\Lambda}_{n_{j_i}}=\prod_{r=1}^{j_i} \#C_r.$$
Therefore
$$\dim_{Be}\overline{\Lambda} \geq \limsup_{i\rightarrow \infty} \frac{\log \#\overline{\Lambda}_{n_{j_i}}}{\log 6M_{n_{j_i}}K_t}=s.$$
Let $h$ be large and choose a integer $R$ so that
$$ M_{n_R}\leq h < M_{n_{R}+1}.$$
Fix $x \in \mathbb{R}$ and consider the counting set 
$$W_{x}:=(x-\frac{1}{2}h, x+\frac{1}{2}h) \cap \overline{\Lambda} $$
each $\lambda \in W_{x}$ has a unique expansion:
$$\lambda=\sum_{j=1}^{\infty} i_j M_{j-1}, \qquad i_j\in \{0,1,\ldots,q_j-1\}$$
We claim that the number of tail sequences 
$$\{i_{n_R+2}i_{n_R+3} \cdots\}$$
is at most two. If three or more such tails existed, we could find $\lambda,\lambda' \in W_x$ with
\begin{align*}
    |\lambda -\lambda'|=&|M_{n_R+1}(\tau_{n_R+1}-\tau_{n_R+1}')+\sum_{j=1}^{n_R}i_jM_{j-1}-\sum_{j=1}^{n_R}i_j'M_{j-1}|\\
    \geq & 2M_{n_R+1}-\sum_{j=1}^{n_R}|i_j-i_j'|M_{j-1}=M_{n_R+1}+M_1>h,
\end{align*}
where $\tau_{n_R+1}=\sum_{j=n_R+2}^{\infty}i_jM_{j-1}M_{n_R+1}^{-1}$. Now we estimate the number of possible occurrences of the coefficients, i.e., $\#\{i_1i_2 \cdots i_{n_R+1}:\lambda \in W_x\}$. 
We know that 
$$\lambda=\sum_{r=1}^{k} (M_{n_{r-1}}c_r+M_{n_r}k(c_r,n_r))=\sum_{j=1}^{\infty} i_j M_{j-1}$$
This implies that 
$$M_{n_1} | i_1+i_2M_1+\cdots +i_{n_1}M_{n_1-1}-c_1$$
Since $\sum_{j=1}^{n_1} i_{j}M_{j-1}, c_1 \in [0, M_{n_1} -1]$, and not all $\overline{\Lambda} $ are contained in $W_x$ then
$$\#\{i_1i_2 \cdots i_{n_1}:\lambda \in W_x\} \leq \#C_1$$
Now fixed $i_1i_2\cdots i_{n_1}$, we can find a $c_1=\sum_{j=1}^{n_1} i_{j}M_{j-1}$, then 
$$\sum_{r=2}^{k} (M_{n_{r-1}}c_r+M_{n_r}k(c_r,n_r))=\sum_{j=n_1+1}^{\infty} i_j M_{j-1}$$
This implies that 
$$M_{n_2} | i_{n_1+1}+i_{n_1+2}M_{n_1+1}+\cdots +i_{n_2}M_{n_2-1}-M_{n_1}c_2$$
Since $\sum_{j=n_1+1}^{n_2} i_{j}M_{j-1}, M_{n_1}c_1 \in [0, M_{n_2} -1]$, and not all $\overline{\Lambda} $ are contained in $W_x$ then
$$\#\{i_{n_1+1}i_{n_1+2} \cdots i_{n_2}:\lambda \in W_x\} \leq \#C_2.$$
proceeding inductively, and taking into account the bounded carry choices
\(0\le k(c_j,n_j)\le K_{a_j}\), we obtain:
\[
\#\{i_{1}i_{2} \cdots i_{n_R}:\lambda \in W_x\}
\leq
\#C_1\#C_2\cdots\#C_R\prod_{j=1}^{R}(1+K_{a_j}).
\]
The levels \(n_R\) are chosen sufficiently fast so that
\[
\frac{\sum_{j=1}^{R}\log(1+K_{a_j})}{\log M_{n_R}}\to0.
\]
Then 
\[
\dim_{Be}(\overline{\Lambda}) \leq
\limsup_{R\rightarrow\infty}
\frac{\log \left(2q_{n_R+1}\#C_1\#C_2\cdots\#C_R\prod_{j=1}^{R}(1+K_{a_j})\right)}{\log M_{n_R}}=s.
\]

\end{proof}

\textbf{Step 5: $\Lambda$ has the cardinality of the continuum.}
At infinitely many stages, use the two choices supplied by Corollary
\ref{cor:3.3} or Lemma \ref{lem:3.5} for one fixed admissible branch, and denote the two resulting marker points by $\xi_r^-$ and $\xi_r^+$, with $\xi_r^-\ne\xi_r^+$. For every binary sequence
\[
I=(i_1,i_2,\ldots)\in\{1,2\}^{\mathbb N},
\]
choose at the $r$-th marker stage the marker $\xi_r^-$ if $i_r=1$ and $\xi_r^+$ if $i_r=2$. Whenever one of the two markers is not chosen, add it to the finite forbidden set used in all later applications of Lemma \ref{lem:3.5}; this is possible because only finitely many markers have been skipped at each finite stage.

If $I\ne J$, let $r_0$ be the first index at which they differ. The marker chosen at stage $r_0$ for $I$ belongs to $\Lambda(I)$. It is not introduced at stage $r_0$ for $J$, and it is put into the later forbidden lists for the construction of $\Lambda(J)$; hence it cannot be produced later by a sparse branch. The dense branches only continue from previously chosen points and cannot create a skipped marker from a different branch. Therefore
\[
\Lambda(I)\ne \Lambda(J).
\]
Hence the family of constructed spectra has cardinality at least $2^{\aleph_0}$, and at most $2^{\aleph_0}$ since all of them are subsets of $\mathbb Z$. Thus its cardinality is continuum.

\begin{proof}[Proof of Theorem \ref{thm:1.3}]
Combining Proposition \ref{prop: 4.1}, Lemma \ref{lem:4.2}, Proposition \ref{prop:4.4} and \textbf{Step 5}, we have fully proved Theorem \ref{thm:1.3}.
\end{proof}

\noindent\textbf{Funding.} the second author was supported in part by the National Key Research and Development Program of China (Nos. 2024YFA1013700), NNSF of China (Nos. 12331005).

\bigskip
\begingroup
\parindent=0pt
\small
\textsuperscript{1}\ \paperaddressone
\par\textit{E-mail address:} \paperemailone
\medskip
\par\textsuperscript{2}\ \paperaddresstwo
\par\textit{E-mail address:} \paperemailtwo\quad ($*$ Corresponding author)
\endgroup

\end{document}